\documentclass{amsart}
\usepackage{geometry}
\usepackage{graphicx}
\usepackage{amsrefs}
\usepackage[colorlinks]{hyperref}
\usepackage{mathpazo}
\usepackage{kotex}
\usepackage{amssymb}

\newtheorem{theorem}{Theorem}[section]

\newtheorem{lemma}[theorem]{Lemma}
\newtheorem{proposition}[theorem]{Proposition}
\theoremstyle{definition}
\newtheorem{definition}[theorem]{Definition}

\numberwithin{equation}{section}

\usepackage{graphicx} 

\begin{document}
\title{The Characterization of Cheng's Multivariate Writhe Polynomial}
\author{Juno Park}
\address{Department of Mathematical Sciences, Korea Advanced Institute of Science and Technology (KAIST), Daejeon 34141, Republic of Korea}
\email{cohomology@kaist.ac.kr}
\begin{abstract}
Cheng \cite{4} introduces the multivariate writhe polynomial, leaving a question regarding the characterization of this invariant. In this paper, we resolve this question by proving that a Laurent polynomial $f$ with integer coefficients can be realized as the multivariate writhe polynomial of some multi-virtual knot if and only if $f(1,\dots,1)=0$ and $\left. \frac{d}{dt} f(t, \dots, t) \right\vert{}_{t=1} = 0$.
\end{abstract}
\maketitle
\section{Introduction}

The concept of virtual knots was first introduced by Kauffman \cite{9}, and their precise geometric interpretation has been studied in \cites{1,13}. Various polynomial invariants have been developed to distinguish virtual knots \cites{2,5,7,8,10,14}. In particular, there have been attempts to generalize the writhe polynomial (or affine index polynomial) to multivariable invariants \cites{4,8}. Specifically, the invariant in \cite{4} is defined for multi-virtual knots \cite{11}, and various studies on invariants for these objects are currently ongoing \cite{12}.

In this paper, we resolve the question raised in \cite{4}*{Section 4.2}. More precisely, we establish the following theorem, which characterizes Cheng's multivariate writhe polynomial.

\begin{theorem} \label{thm:mainA}
A polynomial $f(t_1, \cdots, t_k)\in\mathbb{Z}[t_1^{\pm1}, \cdots, t_k^{\pm1}]$ can be realized as the multivariate writhe polynomial of some multi-virtual knot if and only if $f(1, \cdots, 1)=0$ and $\left. \frac{d}{dt} f(t, \cdots, t) \right\vert{}_{t=1}=0$.
\end{theorem}

As noted in \cite{4}*{Section 4.2}, this can be viewed as an analogue to the following proposition for the writhe polynomial. 

\begin{theorem}[\cites{6,14}]\label{thm1.1}
A polynomial $f(t)\in\mathbb{Z}[t, t^{-1}]$ can be realized as the writhe polynomial of some virtual knot if and only if $f(1)=0$ and $f'(1)=0$.
\end{theorem}

\section{The multivariate writhe polynomial}

Recall the definition of the multivariate writhe polynomial.

\begin{definition}[\cite{4}]
Let $D$ be a multi-virtual knot diagram. The \emph{multivariate writhe polynomial} of $D$ is defined by
\[W_D(t_1, \cdots, t_k)=\sum\limits_{c}w(c)\prod\limits_{i=1}^kt_i^{\operatorname{Ind}_{\alpha_i}(c)}-w(D).\]
\end{definition}

The details of this notation can be found in \cite{4}.

For a multi-virtual knot $K$, let $r(K)$ and $m(K)$ denote the multi-virtual knots obtained from $K$ by reversing its orientation and by switching all classical crossings, respectively.

\begin{proposition}[\cite{4}]\label{proposition2.2}
Let $K$ be a multi-virtual knot. Then we have 
\begin{center}
$W_{m(K)}(t_1, \cdots, t_k)=-W_K(t_1^{-1}, \cdots, t_k^{-1})$ and $W_{r(K)}(t_1, \cdots, t_k)=W_K(t_1^{-1}, \cdots, t_k^{-1})$.
\end{center}
\end{proposition}

\section{The Characterization of the multivariate writhe polynomial}

Now let us prove Theorem~\ref{thm:mainA}.

\begin{lemma}\label{lemma}
    For each positive integer $n$ and $s$, if there exist a multi-virtual knot $K_{n,s}$ satisfying
    \\ $W_{K_{n,s}}(t_1, \cdots, t_{n+s})=t_1 \cdots t_n t_{n+1}^{-1}\cdots t_{n+s}^{-1} - (n-s)t_1 +(n-s-1)$, then the Theorem \ref{thm:mainA} is true.
\end{lemma}
\begin{proof}
The necessity part of Theorem \ref{thm:mainA} follows from \cite{4}*{Proposition 3.5} and Theorem \ref{thm1.1}.

A generating set for $\{f \in \mathbb{Z}[t_1^{\pm 1}, \dots, t_i^{\pm 1},\dots]: f(1, \dots, 1,\dots)=0,\left. \frac{d}{dt} f(t, \dots, t,\dots) \right\vert{}_{t=1}=0\}$ is given by $S=\{t_1^{e_1}\dots t_m^{e_m}-et_1 + (e-1) : (e_1,\dots , e_m)\in \mathbb{Z}^m, e=\sum_{i=1}^m e_i , m=1,2,3,\dots\}$. Let $g(t_1,\dots,t_m)=t_1^{e_1}\dots t_m^{e_m}-et_1 + (e-1)$, $n=1+\sum_{e_i>0} |e_i|$, and $s=1+\sum_{e_i<0} |e_i|$. We consider the knot $K_{n,s}$. By suitably relabeling the virtual crossings of $K_{n,s}$ (which were originally labeled $\alpha_1, \alpha_2, \dots, \alpha_{n+s}$) using the labels $\alpha_1, \dots, \alpha_m$, we can construct a new multi-virtual knot $K$ satisfying $W_K(t_1,\dots, t_m)=g(t_1,\dots,t_m)$. (The idea of this relabeling is to effectively substitute $t_1, t_{i_1}, \dots, t_{i_{n-1}}, t_1, t_{j_1}, \dots, t_{j_{s-1}}$ in order into the variables of $W_{K_{n,s}}(t_1, \dots, t_{n+s})$, for $t_{i_*}, t_{j_*} \in \{t_1, \dots, t_m\}$ such that $t_1^{e_1}\dots t_m^{e_m} = t_1(t_{i_1}\cdots t_{i_{n-1}})t_1^{-1}(t_{j_1}^{-1}\cdots t_{j_{s-1}}^{-1})$.)

One can easily check that a (diagrammatic) connected sum of two multi-virtual knot diagrams adds their $W$. Therefore, by using the connected sum and the operation $m(r(\cdot))$, we can find a multi-virtual knot corresponded  to any element in $\mathrm{span}(S)$ by $W$.
\end{proof}

\begin{proof}[Proof of Theorem~\ref{thm:mainA}]
    By Lemma~\ref{lemma}, we only need to find $K_{n,s}$ for each $n$ and $s$.
    
    First, Figure~\ref{figure1} shows a diagram of $K_{n,0}$. (We can still use the definition of $K_{n,s}$ from Lemma~\ref{lemma} even when $s=0$. Figure~\ref{figure1} is valid for $n \geq 1$.) Its Gauss diagram is shown in Figure~\ref{figure2}, and by calculation, we obtain $W_{K_{n,0}}(t_1, \dots, t_n) = t_1 \dots t_n - nt_1 + (n-1)$.

    There are classical crossings from the 0-th to the $n$-th in Figure~\ref{figure1}. If we define $K_{n,1}$ by changing the 1st crossing into a virtual crossing with a new label $\alpha_{n+1}$ (see Figure~\ref{figure3}), then the Gauss diagram of $K_{n,1}$ is given in Figure~\ref{figure4}. Therefore, $W_{K_{n,1}} = t_1 \dots t_n t_{n+1}^{-1} - (n-1)t_1 + (n-2)$.

    For $n \geq1$ and $1\leq s \leq n$, if we define $K_{n,s}$ by changing the 1st to the $s$-th classical crossings into virtual crossings with labels $\alpha_{n+1}, \dots, \alpha_{n+s}$, then we obtain $W_{K_{n,s}}(t_1, \dots, t_{n+s}) = t_1 \dots t_n t_{n+1}^{-1} \dots t_{n+s}^{-1} - (n-s)t_1 + (n-s-1)$.

    If $s > n \geq 1$, we first consider $r(K_{s,n})$. Its $W$ is $t_1^{-1} \dots t_s^{-1} t_{s+1}\dots t_{s+n} - (s-n)t_1^{-1} + (s-n-1)$. Since $T$ in Figure~\ref{figure5} and $E$ in Figure~\ref{figure6} have $W$ as $t_1+t_1^{-1}-2$ and $t_{s+1}-t_{1}$ respectively, adding $(s-n)$ copies of $T \sharp E$ to $r(K_{s,n})$ via connected sum gives a knot whose $W$ is $t_1^{-1} \dots t_s^{-1} t_{s+1}\dots t_{s+n} - (n-s)t_{s+1} + (n-s-1)$. Now, in this large multi-virtual knot, we relabel $\alpha_1, \dots, \alpha_s, \alpha_{s+1}, \dots, \alpha_{s+n}$ as $\alpha_{n+1}, \dots, \alpha_{n+s}, \alpha_1, \dots, \alpha_n$, respectively. Let $K_{n,s}$ be the resulting knot. Then $K_{n,s}$ satisfies the defining formula of Lemma~\ref{lemma}.
\end{proof}

\begin{figure}[h]
\centering
\includegraphics[width=13.5cm]{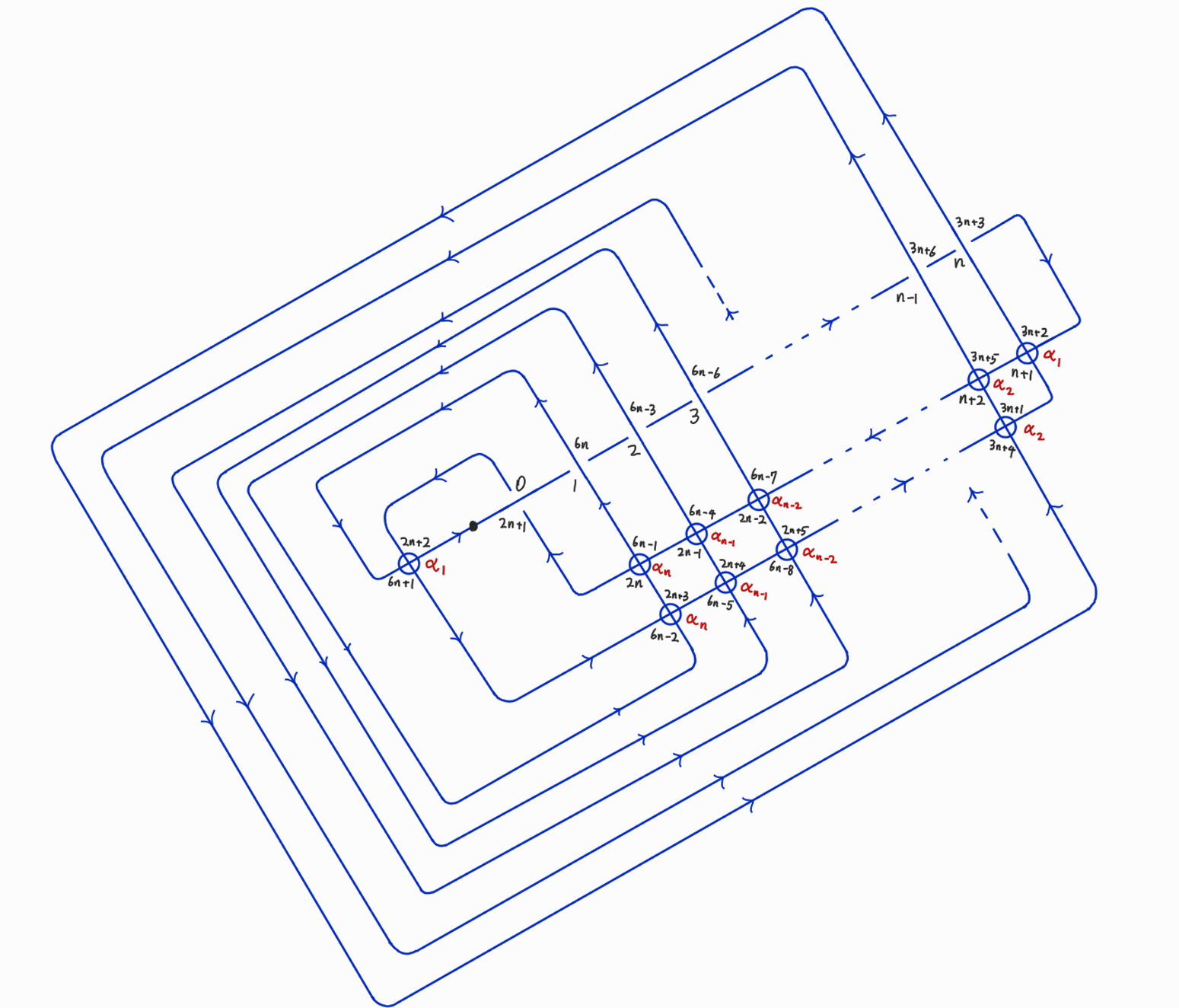}
\caption{Multi-virtual knot diagram of $K_{n,0}$}\label{figure1}
\end{figure}
\begin{figure}[h]
\centering
\includegraphics[width=7.5cm]{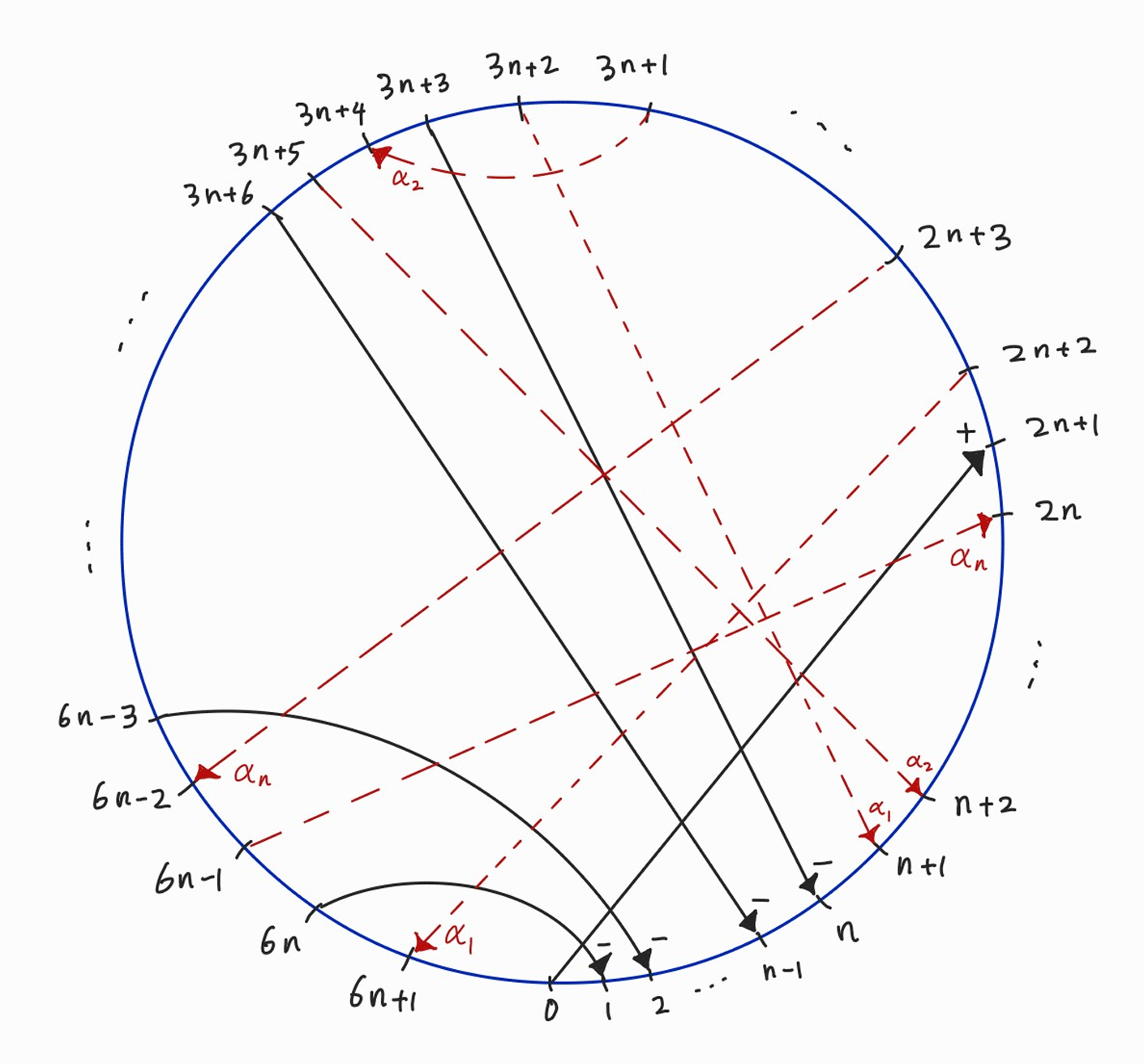}
\caption{Gauss diagram of $K_{n,0}$}\label{figure2}
\end{figure}

\clearpage

\begin{figure}[h]
\centering
\includegraphics[width=12.1cm]{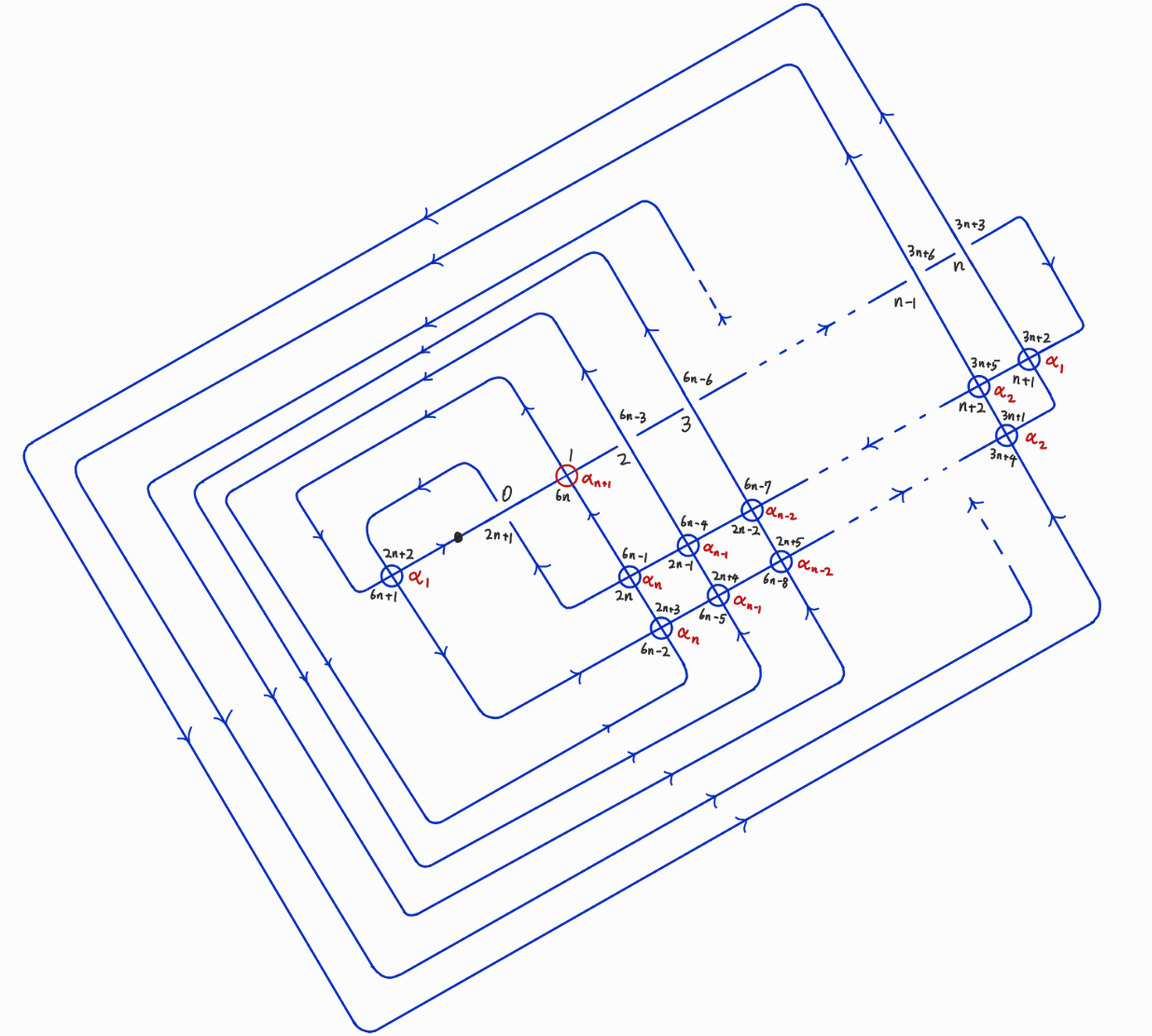}
\caption{Multi-virtual knot diagram of $K_{n,1}$}\label{figure3}
\end{figure}
\begin{figure}[h]
\centering
\includegraphics[width=7cm]{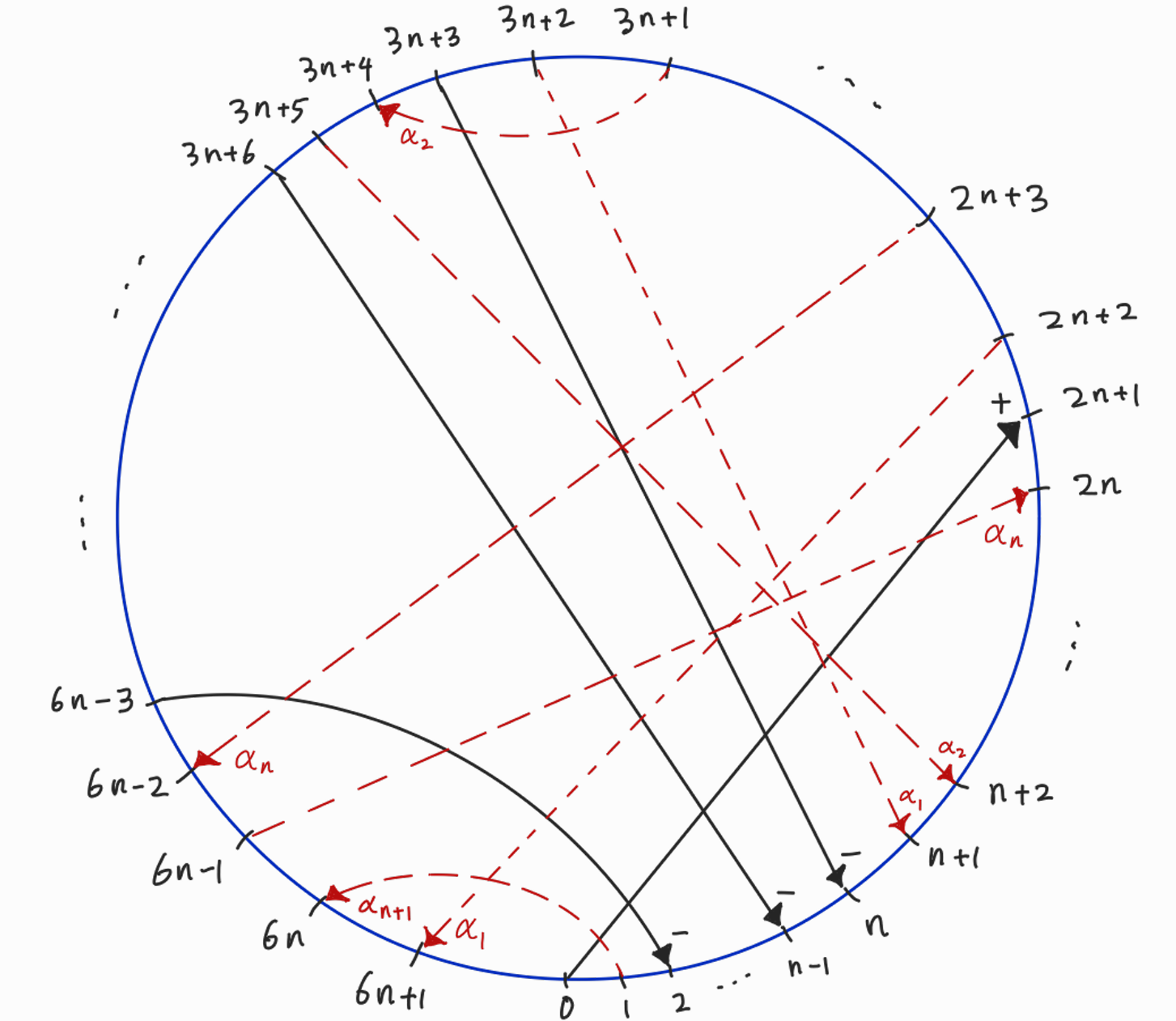}
\caption{Gauss diagram of $K_{n,1}$}\label{figure4}
\end{figure}

\clearpage

\begin{figure}[h]
  \centering
  \begin{minipage}[b]{0.48\textwidth}
    \centering
    \includegraphics[width=5cm]{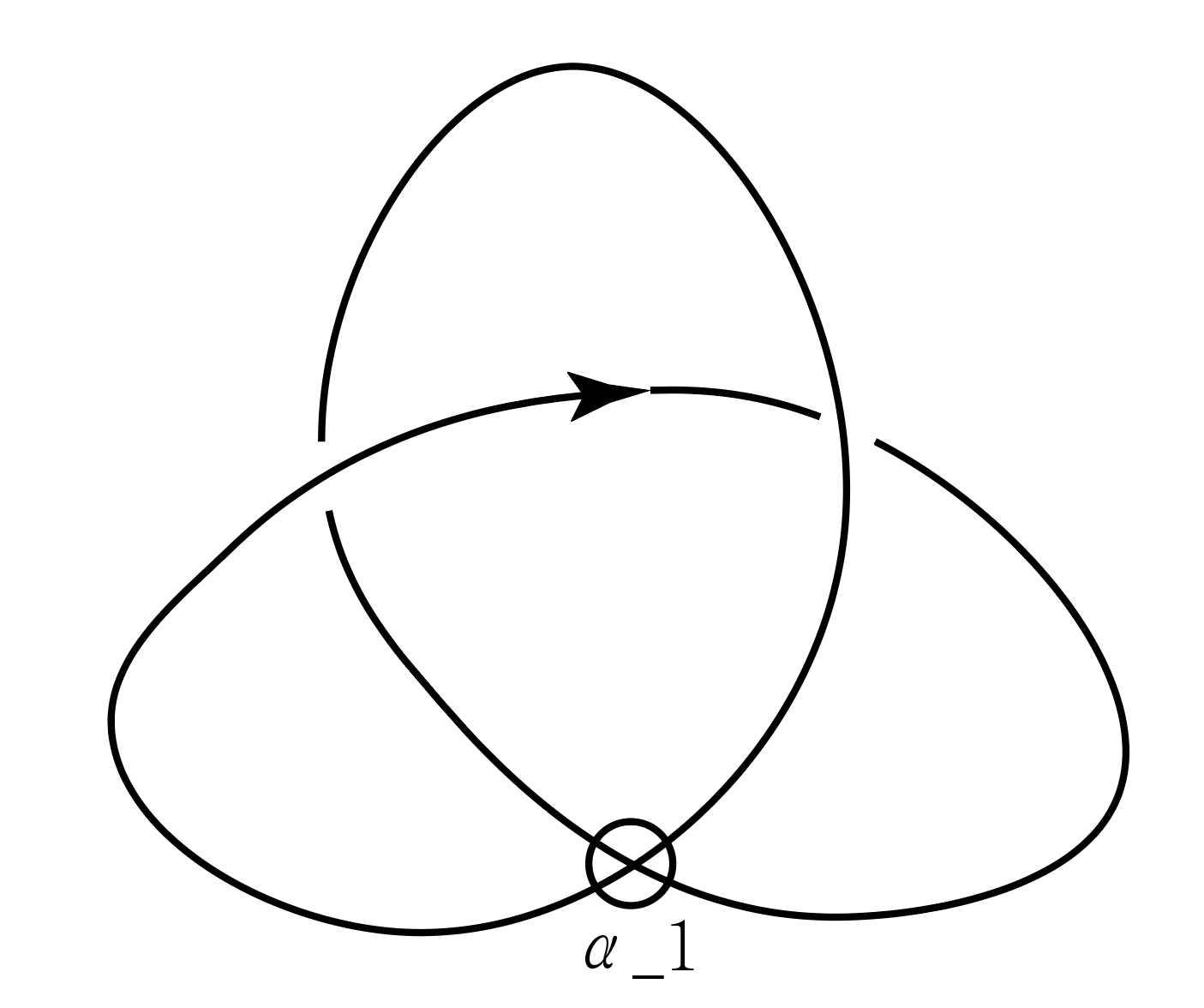}
    \caption{Diagram of $T$}\label{figure5}
  \end{minipage}\hfill
  \begin{minipage}[b]{0.48\textwidth}
    \centering
    \includegraphics[width=7cm]{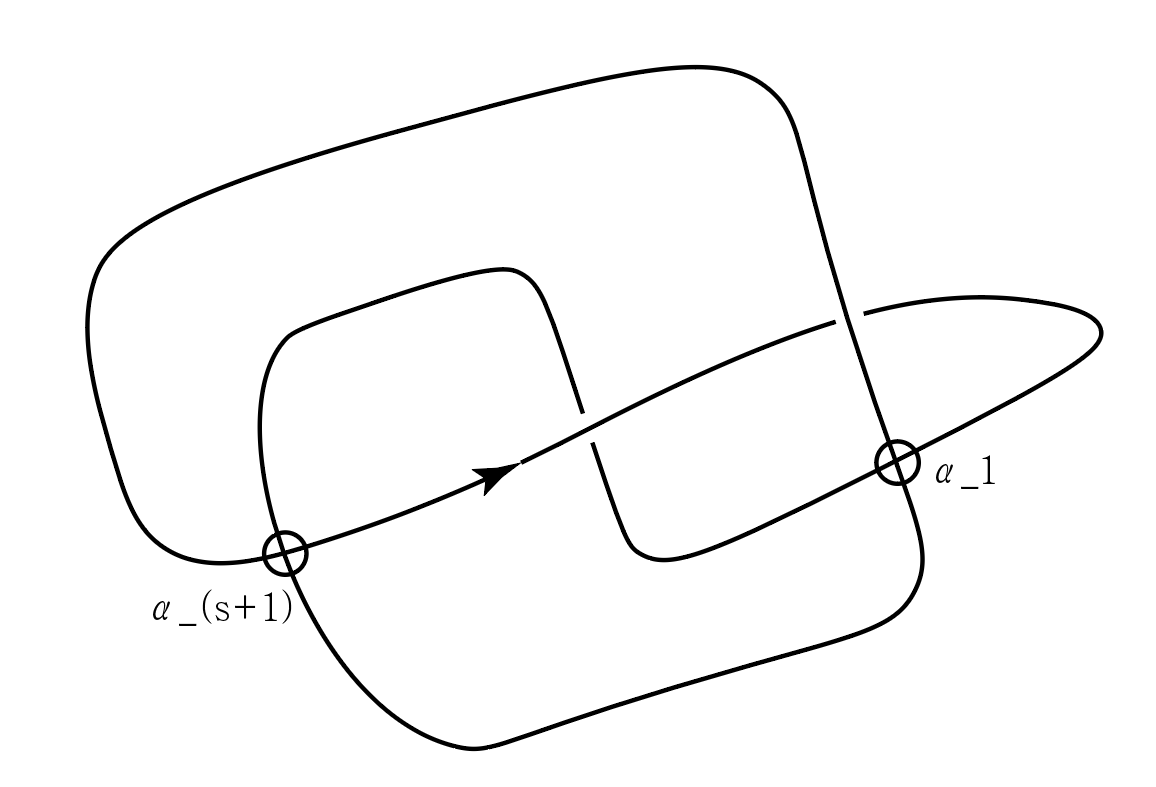}
    \caption{Diagram of $E$}\label{figure6}
  \end{minipage}
\end{figure}

\section*{Acknowledgements}

I am deeply grateful to Dr. Myeong-Ju Jeong for introducing me to knot theory and helping me get interested in doing research in this field.


\begin{thebibliography}{0}



\bib{1}{article}{
	author={Carter, J. Scott},
	author={Kamada, Seiichi},
	author={Saito, Masahico},
	title={Stable equivalence of knots on surfaces and virtual knot cobordisms},
	journal={J. Knot Theory Ramifications},
	volume={11},
	date={2002},
	number={3},
	pages={311--322}}	

\bib{2}{article}{
	author={Cheng, Zhiyun},
	title={A polynomial invariant of virtual knots},
	journal={Proc. Amer. Math. Soc.},
	volume={142},
	date={2014},
	number={2},
	pages={713--725}}


\bib{4}{article}{
	author={Cheng, Zhiyun},
	title={Multivariate writhe polynomial of multi-virtual knots},
	date={2026},
	eprint={arXiv:2606.22501}}

\bib{5}{article}{
	author={Cheng, Zhiyun},
	author={Gao, Hongzhu},
	title={A polynomial invariant of virtual links},
	journal={J. Knot Theory Ramifications},
	volume={22},
	date={2013},
	number={12},
	pages={1341002 (33 pages)}}

\bib{6}{article}{
AUTHOR = {Cheng, Zhiyun},
author={Fedoseev, Denis A.},
author={Gao, Hongzhu},
author={Manturov, Vassily O.},
author={Xu, Mengjian},
TITLE = {From chord parity to chord index},
JOURNAL = {J. Knot Theory Ramifications},
VOLUME = {29},
YEAR = {2020},
NUMBER = {13},
PAGES = {2043004, 26}}


\bib{7}{article}{
	author={Young Ho Im},
	author={Sera Kim},
	author={Dong Soo Lee},
	title={The parity writhe polynomials for virtual knots and flat virtual knots},
	journal={J. Knot Theory Ramifications},
	volume={22},
	date={2013},
	number={1},
	pages={1250133 (20 pages)}}

\bib{8}{article}{
	author={Jeong, Myeong-Ju},
	title={A multivariable polynomial invariant of virtual knots},
	journal={J. Knot Theory Ramifications},
	volume={34},
	date={2025},
	number={6},
	pages={2550025}}

\bib{9}{article}{
	author={Louis H. Kauffman},
	title={Virtual knot theory},
	journal={Europ. J. Combinatorics},
	volume={20},
	date={1999},
	number={},
	pages={663--691}}

\bib{10}{article}{
	author={Louis H. Kauffman},
	title={An affine index polynomial invariant of virtual knots},
	journal={J. Knot Theory Ramifications},
	volume={22},
	date={2013},
	number={4},
	pages={1340007 (30 pages)}}


\bib{11}{article}{
	author={Louis H. Kauffman},
	title={Multi-virtual knot theory},
	journal={J. Knot Theory Ramifications},
	volume={34},
	date={2025},
	number={14},
	pages={Paper No. 2540002, 78}}

\bib{12}{article}{
AUTHOR = {Kauffman, Louis H.},
author={Mukherjee, Sujoy},
author={Vojt\v echovsk\'y, Petr},
TITLE = {Algebraic invariants of multi-virtual links},
JOURNAL = {J. Algebra},
VOLUME = {698},
YEAR = {2026},
PAGES = {493--532}}

\bib{13}{article}{
	author={Kuperberg, Greg},
	title={What is a virtual link?},
	journal={Algebr. Geom. Topol.},
	volume={3},
	date={2003},
	number={},
	pages={587--591}}

\bib{14}{article}{
	author={S. Satoh},
	author={K. Taniguchi},
	title={The writhes of a virtual knot},
	journal={Fundamenta Mathematicae},
	volume={225},
	date={2014},
	number={1},
	pages={327--342}}



    
\end{thebibliography}
\end{document}